\documentclass{article}
\usepackage[utf8]{inputenc}
\usepackage{amsmath,amsfonts,amsthm}
\usepackage{dsfont}
\usepackage{graphicx}
\usepackage{float}
\usepackage{caption}
\usepackage{xcolor}
\usepackage{comment}
\usepackage{todonotes}
\usepackage{hyperref}
\usepackage[english]{babel}
\usepackage{amsthm}
\newtheorem{theorem}{Theorem}
\newtheorem{corollary}[theorem]{Corollary}
\newtheorem{lemma}[theorem]{Lemma}

\usepackage [autostyle, english = american]{csquotes}
\MakeOuterQuote{"}

\newcommand{\Z}{{\mathbb Z}}
\newcommand{\N}{{\mathbb N}}
\newcommand{\Beta}{{\mathbf B}}
\newcommand{\email}[1]{\href{mailto:#1}{\texttt{{\small #1}}}}

\begin{document}
\title{Power-Law Tails in a Fitness-Driven Model for Biological Evolution\footnote{Research performed during \href{http://markov-chains-reu.math.uconn.edu}{Markov Chains REU} (Summer 2018), partially supported by NSA grant H98230-18-1-0044 to Iddo Ben-Ari.}}
\author{Carter Bedsole\\ \email{cbedsole@georgefox.edu}  
\and Iddo Ben-Ari \\ \email{iddo.ben-ari@uconn.edu}
\and 
\and Grace O'Neil\\ \email{goneil@math.columbia.edu} 
}

\maketitle

\begin{abstract}
We study a discrete-time stochastic process that can also be interpreted as a model for a viral evolution. A distinguishing feature of our process is power-law tails due to dynamics that resembles preferential attachment models. In the model we study, a population is partitioned into sites, with each site labeled by a uniquely-assigned real number in the interval  $[0,1]$ known as fitness. The population size is a discrete-time transient birth-and-death process  with probability $p$ of birth and $1-p$ of death. The fitness is assigned at birth according to the following rule: the new member of the population either ``mutates" with probability $r$, creating  a new site uniformly distributed on $[0,1]$  or ``inherits" with probability $1-r$, joining  an existing site with probability proportional to the site's size. At each death event, a member from the site with the lowest fitness is killed. The number of sites eventually tends to infinity if and only if $pr>1-p$. Under this assumption, we  show that as time tends to infinity, the joint empirical measure of site size and fitness (proportion of population in sites of size and fitness in given ranges) converges a.s. to the  product of a modified Yule distribution and  the uniform distribution on $[(1-p)/(pr),1]$. Our approach is based on the method developed in \cite{similar-but-different}. The model and the results were independently obtained by Roy and Tanemura in \cite{roy}.  
\end{abstract}  

\section{Introduction and Statement of Results}

\subsection{Description of the Model}
The motivation for our work is the  paper of Ben-Ari and Schinazi, \cite{similar-but-different} who studied a model for the evolution of a quasispecies. 

Our model is a discrete-time model describing population evolution. The population size $T_n$ at time $n\in\Z_+$ is discrete-time birth and death process on the set of nonnegative integers $\Z_+$, with probability $p\in (0,1)$ of a birth and $1-p$ of a death: 
\begin{align*}  & P( T_{n+1}=T_n +1 | T_0,\dots, T_n) = p,\\
 & P(T_{n+1}=T_n -1| T_0,\dots, T_n) = (1-p){\bf 1}_{\{T_n >0\}}\\
   &P(T_{n+1}=T_n | T_0,\dots, T_n) = (1-p){\bf 1}_{\{T_n =0\}}
\end{align*} 
The event $\{T_{n+1}-T_n =1\}$ is a birth at time $n+1$ while the event $\{T_{n+1}-T_n =-1\}$ is a death at time $n+1$. We consider all members of the population at time $n=0$ as born at time $n=0$. 
 Each member of the population is assigned a fitness value in $[0,1]$ at its birth.
We will call all members of the population sharing a common fitness value a ``site'' at the given fitness value. 

Birth events are split into two different categories: mutation and inheritance. 
\begin{enumerate} 
\item Mutation. With probability $r$, independently of the past, there is a mutation: the new member is assigned a uniformly distributed fitness on $[0,1]$, independently of the past. 
\item\label{case:joing_site}  Inheritance.  With probability $1-r$, if the population size is not zero, the new member joins an existing site, and otherwise the fitness is assigned as in the mutation case. The probability of joining a site is proportional to the number of elements in the site. 
\end{enumerate}
At a death event a member with the lowest fitness is eliminated from the population. Write $U_n^k(f)$ for the number of sites of size $k$ at time $n$ with fitness $\ge f$. 

The main object of our investigation is the the asymptotic behavior of the empirical site-size distribution,  $\widehat {\mu_n^f}$: 
$$ \widehat{\mu_n^f}  (k) = \begin{cases} \delta_{0,k} & \mbox{ if }T_n=0 \\ 
\frac{ k U_n^k (f)}{T_n} & \mbox{otherwise.}\end{cases}$$
In other words, $\widehat{\mu_n^f}(k)$ is equal to the proportion of the population at time $n$ in sites of size $k$ with fitness $\ge f$ if the population is not empty, or is the delta measure at $0$ if the population is empty. Note that the family $\{\widehat{\mu_n}^f(\cdot):f\in [0,1]\}$ determines the joint empirical distribution of fitness and site-size. Specifically, for any  interval  $[a,b)\subset [0,1]$ and $B\subset \N$,
$$ \sum_{k\in B}\left( \widehat{\mu_n^a}(k) -\widehat{\mu_n^b}(k)\right).$$ 
is the proportion of the population in sites with fitness in $[a,b)$ and size in $B$. 

The model is closely  related to the Dirichlet process or the Chinese Restaurant Model \cite{aldous},  but it is fundamentally different because of the constant probability of generating new site as well as the elimination mechanism. In particular, our process is not exchangeable. 

\subsection{A similar model} 
This model is a variation of the one studied in  \cite{similar-but-different}. The differences in that model are:

\begin{itemize}
\item inheritance is uniform among all sites - each existing fitness value is equally likely to be selected at a birth event
\item at a death event the entire site with lowest fitness is eliminated. 
\end{itemize}
Note that this inheritance mechanism is like representation of states in the US Senate, where each state is equally represented, while the model to be studied here is  closer to (an ideal) representation of states in the US House of Representatives, where each state is represented in proportion to its population. 

For the model of \cite{similar-but-different}, the number of sites tends to infinity a.s. if and only if  $pr >1-p$, because the process counting the number of sites is itself a birth and death process. The main result of that paper is that the  random  proportion of sites of a given size and fitness in a certain range converges a.s. to a product  $\mbox{Geom}(\frac{pr-(1-p)}{p-(1-p)})$ and independent uniform of $[f_c,1]$, where $$f_c = \frac{1-p}{pr}.$$

Intuitively:  site sizes remain small and massive sites are exponentially rare, as there is no ``incentive'' to grow large sites.  
 
In our model, sizes of sites do matter, and we study how. Our main result shows that in our case site sizes exhibit a power-law decay. 

\subsection{Our work} 
A large number of real-life phenomena exhibit power-law decay, e.g.  \cite{newman}. In this context, it has two distinguishing features: the identification of sites by their fitness and the possibility of elimination of sites.    

We turn to the analysis of our model. For $f \in [0,1]$, let $T_n (f)$ be the number of members of the population at sites with fitness $\le f$. It immediately follows from the definition that the process $(T_n (f):n\in\Z_+)$ is a birth and death process with probability $pr$ of a birth and $1-p$ of a death. Therefore,  $T_n (f)$ is positive recurrent if $pr<1-p$, null recurrent if $pr=1-p$, and transient if $pr>1-p$. Thus, throughout our  discussion, we will assume \begin{equation} 
\label{eq:transience} 
pr >1-p.
\end{equation}
Under this assumption, define the critical fitness $f_c \in [0,1)$: 
$$f_c = \frac{1-p}{pr}.$$
In particular, any site with fitness $\le f_c$ will eventually be eliminated. 

To state our results, we recall Beta function $\Beta$ and the Yule-Simon distribution: 
$$ \Beta (x,y) = \frac{\Gamma(x)\Gamma(y)}{\Gamma(x+y)}.$$ 
A random variable $Z$ is Yule-Simon distributed with shape parameter $\rho>0$, denoted by $Z\sim \mbox{YS}(\rho)$ if 
$$ P( Z = k) = \rho \Beta (k,\rho+1),~k\in \N,$$
see \cite{wiki} and the references within. Since 
$$P(Z=1) = \rho \Beta (1,\rho+1) = \frac{\rho}{\rho+1}=1-\frac{1}{\rho+1}$$ 
it follows that the distribution of $\tilde Z = Z-1$ conditioned on $\tilde Z >0$ is given by the formula 
$$ P(\tilde Z  =k | \tilde Z >0) = \rho(\rho+1) \Beta (k+1,\rho+1),~k=1,\dots,$$ 
We write $\tilde Z \sim \mbox{YS}_{-1|>0}(\rho)$. As from Stirling's formula, $\Beta(x,y)\sim \Gamma(y) x^{-y}$ as $x\to \infty$, we recover 
 \begin{equation} \label{eq:powerbehavior} P(\tilde Z = k) \sim \rho \Gamma (\rho+2) k^{-\rho-1},~\mbox{ as }k\to\infty.
 \end{equation}
Our main result is the following. This was independently obtained using different methods in \cite{roy}.  
\begin{theorem}
\label{th:main}
Suppose that $pr>1-p$. Let $f>f_c$. Then
\begin{enumerate} 
\item for $k\in\N$, $\displaystyle \lim_{n\to\infty}\widehat{\mu_n^f}(k)=\beta_k$, a.s.  where 
\begin{equation}
\label{eq:the_limit}
\beta_k = \frac{1-f}{1-f_c} P( \mbox{YS}_{-1|>0}(c-1) =k)
\end{equation}
and $c= \frac{2p-1}{p(1-r)}>1$. 
\end{enumerate} 
\end{theorem} 
It therefore follows from \eqref{eq:powerbehavior} that 
 \begin{equation} \label{eq:powerlaw} \beta_k \sim \frac{1-f}{1-f_c} (c-1) \Gamma(c+1)k^{-c},\mbox{ as }k\to\infty.
 \end{equation}
 
The main difficulty in proving the theorem is in showing that the limit in part 1 exists. This will be the main part of the proof.  The key to the proof is to show that the system exhibits a behavior similar to mean reversion. This proof technique is an adaptation of the method developed in \cite{similar-but-different}, and our work is also an illustration of its power, robustness, and applicability for more general evolution mechanisms. Once the limit is established, the particular form \eqref{eq:the_limit} follows directly through soft arguments. Before turning to the proof, we wish to give a heuristic argument for the formula for $\beta_k$ and discuss several additional aspects and corollaries.
 
Here are results of simulations. In all figures, we considered the case $p=0.8$, $r=0.8$ 
\begin{figure}[H]
\centering
	\includegraphics[scale=0.5]{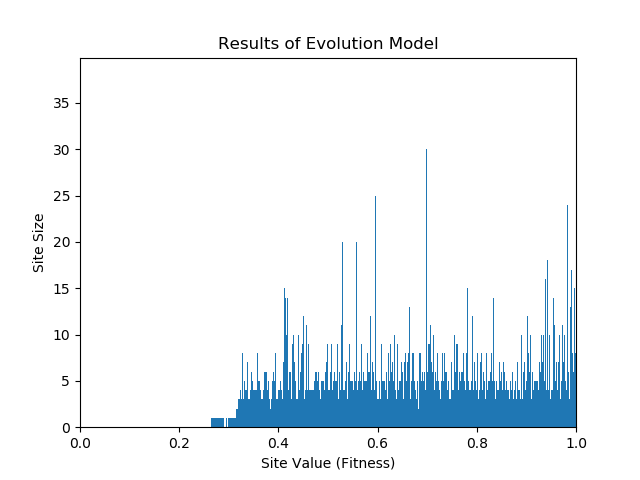}
    \caption{Results of Simulating Model \\ 	$p=0.8,~ r = 0.8,~ n=10^5$}
    \label{sim_results}
\end{figure}
In Figure \ref{sim_results} the x-axis represents the sites fitness values $[0,1]$, and the y-axis represents the number of members with that given fitness value after the $10^5$ steps of the simulation. Notice the sharp drop-off at critical fitness $f_c=\frac{1-p}{pr}=\frac{.2}{.64}=.3125$. No site below the critical fitness value has size $>1$.
In Figure \ref{hist} we arranged data according to site size across all fitness values ($f=0$). In Figure \ref{dontknowwhatproblem} we show the difference between the simulated probabilities in Figure 2 and the theoretical limit probabilities relative to the latter.  
\begin{figure}[H]
	\centering
	\includegraphics[scale=0.6]{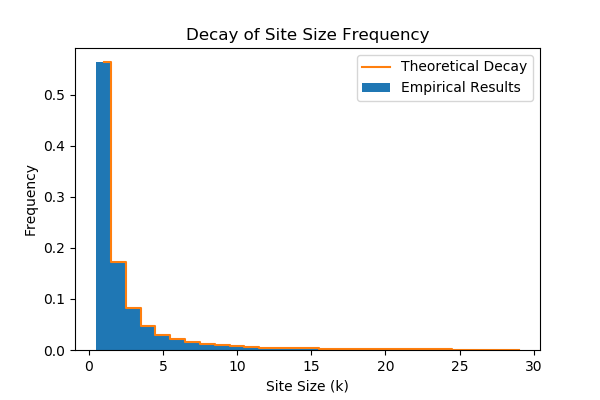}
    \caption{Histogram of Frequency of Site Sizes \\ 	$p=0.8, r = 0.4, n=5*10^6$}
    \label{hist}
\end{figure}

\begin{figure}[H]
	\centering
    \includegraphics[scale=0.6]{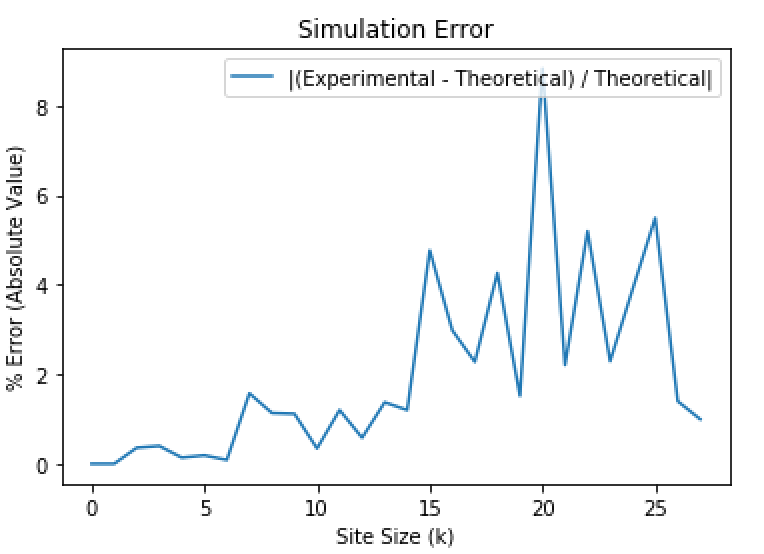}
    \caption{Percent Error of Simulation Results \\ $p=0.8, r = 0.4, n=5*10^6$}
\label{dontknowwhatproblem}
\end{figure}
 As we already noticed, if $f>f_c$, there exists some time $n$ after which no deaths from sites with fitness $f$ or higher will occur. Since we assume $\lim_{n\to\infty}\widehat {\mu_n^f}(1)=\beta_1$, the number of sites of size $1$ with fitness $\ge f$ grows at rate $\gamma \beta_1$. In addition, each step, the number of sites of size $1$ with fitness $\ge f$ will either increase by one in a mutation event or will decrease by one in the case of inheritance. Thus, we expect the following to hold:
\begin{equation} 
\label{eq:beta_equation} \gamma \beta_1 = pr (1-f) - p(1-r) \beta_1.
\end{equation}
That is 
$$ \beta_1 = (1-f)\frac{pr}{\gamma + p(1-r)}.$$ 
As $c= \frac{\gamma}{p(1-r)}$,
$$ 1+c  = \frac{\gamma + p(1-r)}{p (1-r)},$$ 
so we conclude with 
$$ \beta_1 = (1-f) \times \frac{1}{1+c}\times \frac{r}{1-r}.$$ 

We continue to sites of large size. Suppose that $\lim_{n\to\infty} \widehat{\mu_n}(k-1)$ exists and is equal to $\beta_{k-1}$. We introduce a new quantity, $\alpha_k$, equal to the asymptotic growth rate of the number of sites of size $k$ with fitness $\ge f$. Then $k \alpha_k$ gives the rate of growth of part of the population in sites of size $k$ with fitness $\ge f$, and since the growth of the population is at rate $\gamma$, we have 
$$ \beta_k = \frac{ k \alpha_k}{\gamma}.$$ 

To find a general formula for $\beta_k$, observe that the number of sites of size $k$ with fitness $\ge f$ will increase by $1$ in case of a mutation of a site of size $k-1$, or will decrease by $1$ in case of a mutation of site of size $k$. This leads to 
\begin{equation} 
\label{eq:betak_equation} 
 \alpha_k = p(1-r) \left ( \beta_{k-1} - \beta_{k}\right).
\end{equation} 
We draw several conclusions from this equation. First we obtain an explicit expression for $\beta_k$. Since $\beta_k = \frac{k \alpha_k}{\gamma}$, it follows that 
$$ \frac{ \gamma \beta_k}{k} = p(1-r) (\beta_{k-1} - \beta_k).$$ 
Therefore, 
$$ \frac{ \beta_{k-1}}{\beta_k } = \frac{p(1-r) + \gamma/k}{ p(1-r)} = 1 + \frac{c}{k}.$$
This leads to the following formula: 
\begin{equation} 
\begin{split} 
\beta_k & = \beta_1 \prod_{j=2}^k (1+\frac{c}{j})^{-1}\\ 
 & = (1-f) \frac{r}{1-r} \frac{1}{1+c}\prod_{j=2}^k (1+\frac{c}{j})^{-1}\\
 & = (1-f) \times \frac{r}{1-r} \frac{ k! }{\prod_{j=2}^k (j+c)}\\
  & = (1-f) \frac{ r }{1-r} \frac{ \Gamma(c+1)\Gamma (k+1) }{\Gamma (c+k+1)}\\
  & = (1-f) \frac{r}{1-r} \frac{c}{c-1} (c-1)\Beta(k+1,c)\\ 
  & = (1-f) \frac{r}{(1-r)(c-1)}P( \mbox{YS}_{-1|>0}(c-1)=k).
\end{split} 
\end{equation} 
Next, observe that 
$$\frac{1}{c-1} = \frac{p(1-r)}{2p-1-p+pr }=\frac{p(1-r)}{p(1+r)-1},$$ 
so that 
$$\frac{r}{(1-r)(c-1)}=\frac{ p(1-r)}{p(1+r)-1}\frac{r}{1-r}= \frac{pr}{p(1+r)-1}=\frac{1}{1-f_c},$$ 
giving \eqref{eq:the_limit}. 
 Thus we obtain the following tightness result: 
\begin{corollary} 
 Suppose $pr>1-p$. Then the empirical site size distribution is tight in the sense that for any $\epsilon>0$ there exists $M\in\N$ and $f>f_c$ such that $ \widehat{\mu_{n}^f} (\{1,\dots,M\})>1-\epsilon$.
 \end{corollary} 
 The tightness is important because it shows that no mass spills to $+\infty$: the proportion of population in sites larger than $M$ eventually converges to a constant which tends to $0$ as $M\to\infty$. The tightness also allows us to look at proportion of sites with given sizes rather than proportions of members of the population in sites of given sizes. For each $k$, the proportion of sites of size $k$ with fitness $\ge f$ is equal to 
 $$\frac{U_n^k(f)}{\sum_{j=1}^\infty U_n^j(f)}.$$
This proportion can be rewritten as 
$$ \frac{T_n \widehat{\mu_n^f(k)/k}}{T_n \sum_{j=1}^\infty \widehat{\mu_n^f}(j)/j}.$$ 
By the tightess, for every $\epsilon>0$, there exists $M$ such that $\sum_M^\infty \widehat{\mu_n^f}(j)/j<\epsilon$, we obtain
\begin{corollary}
The proportion of sites of size $k$ with fitness $\ge f$ converges a.s. as $n\to\infty$ to $C\beta_k/k$, where $C=(\sum_{k=1}^\infty \frac{\beta_k}{k})^{-1}$. 
\end{corollary}

Next, in light of \eqref{eq:sum},  rewrite the expression for $\beta_k$, by dividing and multiplying by $1-f_c$. This gives: 
$$\beta_k  = \frac{1-f}{1-f_c}\times (1-f_c) \frac{r}{1-r}c.$$ 
But  $$(1-f_c)\frac{r}{1-r} = \frac{pr - (1-p)}{p(1-r)}=\frac{p(r-1)+p-(1-p)}{p(1-r)}=c-1,$$
and therefore 
$$ \beta_k = \frac{1-f}{1-f_c}\times (1-f_c) \frac{r}{1-r}c \Beta(k+1,c)  = \frac{1-f}{1-f_c}\times \rho_k,$$ 
where $$\rho_k =(c-1) c \Beta(k+1,c).$$
But by  \eqref{eq:sum}, 
$$ \sum_{k=1}^\infty \rho_k = 1.$$
Let $\Beta_+(c)$ be the probability distribution with PMF $(\rho_k:k\in \N)$. Then we proved the following 
\begin{corollary}
The empirical distribution of site size and fitness of an individual converges a.s. to the product of $\Beta_+(c)$ and $U[f_c,1]$.   
\end{corollary}

\section{Proof of main result} 
Before we prove the theorem, we will formally construct the model.From three independent  sequences of IID random variables, $(X_n:n\in\mathds{N})$, $(R_n:n\in\mathds{N})$, and $(V_n:n\in\mathds{N})$ where $X_i \sim Bern(p)$, $R_i \sim Bern(r)$, and $V_n \sim U[0,1]$. For $n \in \mathds{Z}_+$, recall that $T_n$ represents the total population at time $n$, $T_n$. Let $U_n^k(f)$ denote the number of sites of size $k$ with fitness $\ge f$ at time $n$. 
To define our process from the three IID sequences, we proceed as follows: 
\begin{enumerate}
\item If $X_{n+1} = 1$, then the population, $T_n$ increases by one individual.
	\begin{enumerate}
    \item If $R_{n+1} = 1$ or if $T_n = 0$, then the fitness of the new individual is $V_{n+1}$.
    \item Otherwise $R_{n+1} = 0$ and the new individual selects an existing individual and inherits its fitness. The existing individual is chosen uniformly among existing individuals; this creates a weighted average so that more populous sites are more likely to be selected. This can be done by partitioning the interval $[0,1]$ into $T_n$ intervals all of equal length. Each interval will be labeled by the fitness of each individual alive at time $n$. Then, the selected interval will correspond to the interval to which $V_{n+1}$ belongs.
    \end{enumerate}
\item If $X_{n+1} = 0$ and $T_n > 0$, then an individual from the site with the lowest fitness is removed from the population.
\end{enumerate}

Let $\gamma = p-(1-p)$ denote the local drift of the birth and death process governing the population size. 


Now that we have formally constructed the process, we will briefly explain the argument we will use to prove Theorem \ref{th:main}. Our argument is based on the technique developed in \cite{similar-but-different}, adapted to the present case. The technique utilizes the fact both models exhibit a mean-reverting behavior pushing the system towards its expected value. The devil is in the details... 

Since we're dealing with ratios of random quantities, concentration will greatly simplify the analysis. This will be provided by the following large deviations lemma which is standard. We will give the proof in the Appendix.
\begin{lemma}
\label{lem:LDP}
Fix $\epsilon > 0$. Then there exist positive constants $\zeta,c>0$ such that for any $n \in \N$ 
\begin{equation}
P\left(\left|\frac{T_n}{n} - \gamma\right| > \epsilon\right)\leq ce^{-n\zeta}
\end{equation}
\label{boundLemma}
\end{lemma}

\begin{proof}
\subsubsection*{Base case $k=1$.}
To simplify notation we omit the dependence on $k=1$ and on $f$, writing $\beta$ for $\beta_1$, $U_n$ for $U_n^1(f)$, etc. We recall that $\gamma= p -(1-p)$ is the asymptotic population growth rate. Also, in light of the heuristic \eqref{eq:beta_equation}, let 
$$\alpha =\gamma \beta= pr (1-f) - p(1-r)\beta.$$ 
Fix $\epsilon > 0$. Using the convention $\inf \emptyset = \infty$, define the following sequence of stopping times.
\begin{equation}
\begin{split}
\sigma _1 = \inf\{n\in\N:\frac{U_n}{T_n} - \beta \in (2\epsilon, 3\epsilon)\}
\\
\tau_{l}^u = \inf\{n>\sigma_{l}:\frac{U_n}{T_n}> \beta + 4\epsilon \}
\\
\tau_{l}^d = \inf\{n>\sigma_{l}: \frac{U_n}{T_n}< \beta + \epsilon \}. 
\end{split}
\label{stoppingTimes1}
\end{equation}
Thus, $\sigma_1$ is the first time that $U_n/T_n$ is in the band between $2\epsilon$ and $3\epsilon$ above $\beta$. Once $\sigma_l$ has been defined $\tau_l^u$ and $\tau_l^d$ are, respectively, the first time after $\sigma_l$ that $U_n/T_n$ drifts $\epsilon$ above the band (above $\beta + 4\epsilon$) or $\epsilon$ below the band (below $\beta+\epsilon$). Let 
$$\tau_{l} = \min\{\tau_{l}^u, \tau_{l}^d\},$$ and continue defining $\sigma_{l+1}$ through 
\begin{equation}
\sigma_{l + 1} = \inf\{n>\tau_{l} : \frac{U_n}{T_n} - \beta \in (2\epsilon, 3\epsilon)\}.
\end{equation}
On the event $\{\sigma _{l}<\infty\}$, and for nonnegative integer $m$ we define
\begin{equation}
Z_m = X_{\sigma_{l} + m}R_{\sigma_{l} + m}\textbf{1}_{\{V_{\sigma_{l} + m} \geq f \}}-X_{\sigma_{l} + m}(1-R_{\sigma_{l} + m})\textbf{1}_{\{1-V_{\sigma_{l} + m} \leq \beta + \epsilon \}}.
\end{equation}
 From the definition,
\begin{equation}
\label{eq:Z_expectation}
E[Z_m|\sigma_l < \infty] = pr(1-f)-p(1-r)(\beta + \epsilon) < pr(1-f)-p(1-r)\beta = \alpha.
\end{equation}
Note that $Z_m$ has no reference to the total population process. Yet, when  $1<m<\tau_l-\sigma_l+1$, the proportion of the population at time $\sigma_l+m-1$ in sites of size $1$ is at least $\beta +2\epsilon$, and therefore an inheritance event at time  $\sigma_l+m$ the number of sites of size $1$ will reduce by $1$ with probability larger than $\beta+2\epsilon$. Furthermore, we do not incorporate death events when defining  $Z_m$, and as a result $U_{\sigma_l+m}-U_{\sigma_l+m-1}\le Z_m$. We use the sequence $(Z_m:m\in\Z_+)$ as the increments of a random walk $(\Lambda_n:n\in\Z_+)$: 
$$\Lambda_n = U_{\sigma_l} + \sum_{m\leq n}Z_m.$$ 
And from the discussion above, $\Lambda_n \geq U_{\sigma_{l}+n}$ for $n<\tau_l+2-\sigma_l$. Define:
\begin{equation}
\begin{split}
\tau^{Z, u} = \inf\{n: \Lambda_{n} \geq (\beta +4\epsilon)T_{\sigma_l +n}\}
\\
\tau^{Z, d} = \inf\{n: \Lambda_{n} \leq (\beta +\epsilon)T_{\sigma_l +n}\},
\end{split}
\end{equation}
and so, $\tau^{Z,u}\le \tau^u_l$ and $\tau^{Z,d}\ge \tau^d_l$. This leads to the following sequence of inequalities: 
\begin{equation}
\label{eq:what_we_need_to_bound}
P(\tau_l^u < \tau_l^d, \sigma_l < \infty) \leq P(\tau^{Z,u}<\tau^{Z,d}, \sigma_l < \infty) \leq P(\tau^{Z,u} < \infty, \sigma_{l} < \infty).
\end{equation}
As $\Lambda_0=U_{\sigma_l}$ and $U_{\sigma_l}\le (\beta +3\epsilon) T_{\sigma_l}$, we obtain 
\begin{equation}
\begin{split}
\tau^{Z,u} &= \inf\{n:\Lambda_n - \Lambda_0 \geq (\beta + 4\epsilon)T_{\sigma_l + n} - U_{\sigma_l}\} 
\\
&\geq \inf\{n:\Lambda_n-\Lambda_0 \geq (\beta +4\epsilon)T_{\sigma_l + n} - (\beta +3\epsilon)T_{\sigma_l}\}.
\end{split}
\end{equation}

For $\rho>0$, define 
\begin{equation}
\label{eq:arho}
A_n(\rho) = \{\sup_{m \geq n}\left|\frac{T_m}{m} - \gamma\right|\geq\rho\}.
\end{equation}

From Lemma \ref{boundLemma} and the union bound, it follows that there exist $c$,$\zeta > 0$, such that for all $n \in \mathds{Z}_{+}$
\begin{equation}
P(A_n(\rho)) \leq ce^{-n\zeta}.
\end{equation}

By definition, $\sigma_l \geq l$. Therefore on $A_l^c(\rho)\cap \{\sigma_l<\infty\}$, we have $T_{\sigma_l} < (\gamma + \rho)\sigma_l$ and $T_{\sigma_l+n} \geq (\gamma-\rho)(\sigma_l + n)$ for all $n \in \mathds{Z}_+$. On this event, we can  substitute these inequalities for $T_{\sigma_l}$ and $T_{\sigma_l+n}$ to obtain 
\begin{equation}
\begin{split}
\tau^{Z,u} &\geq \inf\{n: \Lambda_n-\Lambda_0 \geq (\beta + 4\epsilon)(\gamma - \rho)(\sigma_l + n) - (\beta + 3\epsilon)(\gamma + \rho)\sigma_l\}\\
& = \inf\{n:\Lambda_n-\Lambda_0 \geq \epsilon' \sigma_l +\alpha' n \},
\end{split}
\end{equation}
where 
\begin{align*}
&\epsilon'= (\beta + 4\epsilon)(\gamma - \rho) - (\beta + 3\epsilon)(\gamma + \rho),\mbox{ and} \\
&\alpha' = (\beta+4\epsilon)(\gamma-\rho).
\end{align*} 
By choosing $\rho =\rho(\epsilon)$ sufficiently small, we can guarantee that 
both $ \epsilon' >0$ and $\alpha'>\alpha= \beta\gamma$.
With such a choice,  on $A_l^c(\rho) \cap \{\sigma_l < \infty \}$ we have
\begin{align}
\begin{split}
\tau^{Z,u} &\geq \inf\{n:\Lambda_n - \Lambda_0 - \epsilon'\sigma_l \geq \alpha'n\}
\\
&\geq \inf\{n:\Lambda_n - \Lambda_0 - \epsilon'l \geq \alpha'n\}
\\
&\geq \inf \{ n\geq \epsilon'l:\Lambda_n - \Lambda_0 \geq \alpha'n\}
\end{split}
\end{align}

Since by \eqref{eq:Z_expectation}, the random walk $\Lambda$ has increments with with an expectation smaller than $\alpha$ and by choice of $\rho$,  $\alpha < \alpha'$, the large deviations argument from Lemma \ref{boundLemma} guarantees that
\begin{equation}
P(\tau^{Z,u}<\infty, A_l^c(\rho), \sigma_l < \infty) \leq P\left(\sup_{n \geq \epsilon'l}\frac{1}{n}(\Lambda_n-\Lambda_0) > \alpha'\right) < ce^{-\zeta \epsilon' l}.
\end{equation}
This and \eqref{eq:what_we_need_to_bound} give 
$$ P(\tau_l^u < \tau_l^d, \sigma_l < \infty) \le P(A_l(\rho) ) + P( \tau^{Z,u}<\infty, \sigma_l < \infty, A_l^c(\rho)),$$
and it follows from Borel-Cantelli that 
$$P(\{\sigma_l < \infty\} \cap \{\tau_l^u < \tau_l^d\}\mbox{  i.o.}) = 0.$$ 
Therefore, with probability $1$, either 
\begin{enumerate} 
 \item $\{\sigma_l < \infty\}$ finitely often; or 
 \item $\{\sigma_l < \infty\}$ infinitely often and $\{\tau_l^u < \tau_l^d\}$ finitely often. 
 \end{enumerate}
We will examine both cases and will show that for each case 
$$\limsup_{n\to\infty} \frac{U_n}{T_n}\le \beta + 4\epsilon.$$
This allows to conclude that 
\begin{equation} 
\label{eq:upper_one} \limsup_{n\to\infty} \frac{U_n}{T_n}\le \beta,
\end{equation} 
because $\epsilon$ is arbitrary. Before examining the cases, we will now show how the proof follows from \eqref{eq:upper_one}. Indeed, plugging \eqref{eq:upper_one} into \eqref{eq:construction_level_one} and applying the law of large numbers, we obtain 
$$ \liminf_{n\to\infty} \frac{U_n}{n} \ge pr (1-f) - p(1-r) \beta.$$ 
But since $\lim_{n\to\infty} T_n /n = \gamma = p -(1-p)$, we conclude, 
\begin{equation}
       \label{eq:the_inversion_trick}
       \liminf_{n\to\infty} \frac{U_n}{T_n} \ge \frac{pr (1-f)-p(1-r) \beta}{\gamma}= \frac{\alpha}{\gamma}=\beta.
\end{equation}
It remains to complete the analysis of cases 1. and 2. above. We first observe that any case leading to $\liminf_{n\to\infty} \frac{U_n}{T_n} > \beta$ automatically leads to a contradiction due to the argument that lead to \eqref{eq:the_inversion_trick}. Therefore we necessarily have $\liminf_{n\to\infty} U_n/T_n \le \beta$. 

Now for case 1. By the last remark, we know that $U_n/T_n$ spends infinitely many times below $\beta+2\epsilon$. If $\limsup_{n\to\infty} U_n/T_n \ge \beta+4\epsilon$, then each time we cross from below $\beta+2\epsilon$ to above $\beta+4\epsilon$ we must cross through $(\beta+2\epsilon,\beta+3\epsilon)$, contradicting the assumption in this case that $\sigma_l<\infty$ for only finitely many $l$-s. Thus we have established that in case 1, $\limsup_{n\to\infty}U_n/T_n \le \beta +4\epsilon$. 

We turn to case 2. Here we must have $\limsup_{n\to\infty} U_n/T_n \ge \beta+ 4\epsilon$, but as we also must have $\liminf_{n\to\infty} U_n/T_n \le \beta$, and the same argument of crossing in the last paragraph gives that we necessarily have $\tau^d_l<\tau^u_l$ infinitely many times, a contradiction to the definition of the case. 
\begin{equation}
\label{eq:construction_level_one}
U_{n+1} - U_n = X_{n+1}R_{n+1}{\bf 1}_{\{V_{n+1}\ge f\}}-X_{n+1}(1-R_{n+1}){\bf 1}_{\{V_{n+1}\le U_n /T_n\}}
\end{equation} 
\subsubsection*{Induction Step.}
The induction step is essentially a repetition of the argument for the case $k=1$ {\it mutatis mutandis}. We fix $f>f_c$, and as in the case $k=1$ we drop dependence on $f$ in our notation. However, we will not drop dependence on $k$. 
Our induction hypothesis is that the proportion of population 
in sites fitness $\ge f$ and size $k$ is 
$$ \lim_{n\to\infty} \frac{k U_n^k}{T_n} = \beta_k, \mbox{ a.s.}$$ 
Equivalently, the number of sites of size $k$ with fitness $\ge f$ grows at a rate $\alpha_k$
$$ \lim_{n\to\infty}\frac{U_n^k}{n}=\alpha_k, \mbox{ a.s.},$$ 
where $$\alpha_k = \frac{\gamma \beta_k }{k}.$$ 
a.s. 

Again, we construct a process that represents the difference from one timestep of $U^{k+1}$ to another. Since death events occur only finitely often for sites with fitness $\ge f$ and because $\frac{ k U_n^k}{T_n} \to \beta_k $ a.s., we can construct the process so that 
\begin{equation}
\label{kStep}
\begin{split} 
U_{j+1}^{k+1}-U_{j}^{k+1} &\leq X_{j+1}\left(1-R_{j+1}\right)\left(\textbf{1}_{\{V_{j+1}\leq \frac{ k U_j^k}{T_n} \}}-\textbf{1}_{\{1-V_{j+1}<\frac{(k+1)U_{j}^{k+1}}{T_j}\}\}}\right)\\
 &\le X_{j+1}\left(1-R_{j+1}\right)\left(\textbf{1}_{\{V_{j+1}\leq \beta_{j}+\epsilon\}}-\textbf{1}_{\{1-V_{j+1}<\frac{(k+1)U_{j}^{k+1}}{T_j}\}}\right)
\end{split} 
\end{equation}
for all $n$ large enough. We then define a sequence of stopping times similar to (\ref{stoppingTimes1}): 
\begin{equation}
\begin{split}
\sigma_{1}^{k+1} = \inf\{n: \frac{(k+1)U_n^{k+1}}{T_n} - \beta_{k+1} \in (2\epsilon, 3\epsilon), \frac{kU_n^k}{T_n} < \beta_k + \epsilon\}
\\
\tau_l^{u, k+1} = \inf\{n>\sigma_l^{k+1}: \frac{(k+1)U_n^{k+1}}{T_n}>\beta_{k+1} + 4\epsilon\}
\\
\tau_l^{d, k+1} = \inf\{n>\sigma_l^{k+1}: \frac{(k+1)U_n^{k+1}}{T_n}<\beta_{k+1} + \epsilon\}
\end{split}
\end{equation}

As in the base case, we set $\tau_l^{k+1} = \min\{\tau_l^{u, k+1}, \tau_l^{d, k+1}\}$.  Additionally, define 

\begin{equation}
\theta_l^{k+1} = \inf\{n>\sigma_l^{k+1}: \frac{kU_n^{k}}{T_n}>\beta_k +\epsilon\}
\end{equation}

and inductively define 

\begin{equation}
\sigma_{l+1}^{k+1} = \inf\{n>\tau_l^{k+1}:\frac{(k+1)U_n^{k+1}}{T_n}-\beta_{k+1} \in (2\epsilon, 3\epsilon), \frac{kU_n^{k}}{T_n}<\beta_k+\epsilon\}
\end{equation}

By the induction hypothesis, we have that $\{\theta_l^{k+1} < \infty\}$ only finitely often. This sequence of times allows us to only compare the process when the increments satisfy the bound in (\ref{kStep}). 
Consider the event $\{\sigma_l^{k+1}< \infty\}$. Then for $m \in \mathds{Z}_+$, let
\begin{equation}
Z_m^{k+1} = X_{\sigma_l^{k+1}+m}\left(1-R_{\sigma_l^{k+1}+m}\right)\left(\textbf{1}_{\{V_{\sigma_l^{k+1}+m}\leq \beta_k+\epsilon\}}-\textbf{1}_{\{1-V_{\sigma_l^{k+1}+m}\le \beta_k +\epsilon\}}\right)
\end{equation}

From this, we find that for all $m < \min{\{\tau_{l+1}^{k+1}, \theta_l^{k+1}\}} - \sigma_l^{k+1}$, we have $U_{\sigma_l^{k+1}+m}^{k+1} - U_{\sigma_l^{k+1}}^{k+1} \leq Z_m$. Moreover, from the definition of $\alpha_k$, \eqref{eq:betak_equation}, we have 
\begin{equation}
E[Z_m|\sigma_l<\infty] 
= p(1-r)(\beta_k+\epsilon - (\beta_{k+1}+\epsilon)) = \alpha_{k+1}
\end{equation}
On $\{\sigma_l^{k+1} < \infty\}$ we can define a random walk $\Lambda_n = U_{\sigma_l^{k+1}}^{k+1} + \sum_{m\leq n}Z_m$, and so for all $n < \min\{\tau_l^{k+1} + 1, \theta_l^{k+1}\} - \sigma_l^{k+1}$, we have $\Lambda_n\geq U_{\sigma_l^{k+1}+n}^{k+1}$. Letting

\begin{equation}
\begin{split}
&\tau^{Z, u} = \inf\{n:\Lambda_n\geq(\beta + 4\epsilon)T_{\sigma_{l}+n}\mbox{ and } 
\\
& \tau^{Z,d} = \inf\{n:\Lambda_n < (\beta + \epsilon)T_{\sigma_{l}+n}\},
\end{split}
\end{equation}
and so we find 
\begin{equation}
\begin{split}
P(\tau_l^{u, k+1} < \tau_l^{d, k+1}&  < \theta_l^{d, k+1}, \sigma_l^{k+1} < \infty) \\
& \leq P(\tau^{Z, u}<\tau^{Z, d}, \sigma_l^{k+1} < \infty)
\\
&\leq P(\tau^{Z, u}< \infty, \sigma_l^{k+1} < \infty).
\end{split}
\end{equation}
Repeating the large deviations and Borel-Cantelli argument following \eqref{eq:arho} leads to the conclusion that 
$$ P(\{\sigma_l^{k+1}<\infty\} \cap \{\tau_l^{u,k+1} < \tau_l^{d,k+1}\} \mbox{ i.o. })=0.$$ 
As before, it is enough to show that 
$$ \limsup_{n\to\infty} \frac{(k+1) U_{n}^{k+1}}{T_n}\le \beta_{k+1}.$$ 
Indeed, if this holds, then from the law of large numbers, 
$$ \liminf_{n\to\infty} \frac{ (k+1) U_{n}^{k+1}}{T_n} \ge \frac{k+1}{\gamma } p(1-r)\left ( \beta_k - \beta_{k+1}\right ) = \frac{ (k+1) \alpha_{k+1}}{\gamma}= \beta_{k+1},$$ 
completing the proof. 

If $\liminf_{n\to\infty} \frac{(k+1)U_n^{k+1}}{T_n} > \beta_{k+1}$, then it follows from the law of large numbers that 
$$ \limsup_{n\to\infty} \frac{(k+1)U_n^{k+1}}{T_n} < \frac{k+1}{\gamma} p(1-r) \left ( \beta_k - \beta_{k+1}\right)=\beta_{k+1}, $$
a contradiction. Therefore we necessarily have $\liminf_{n\to\infty} \frac{(k+1)U_n^{k+1}}{T_n} \le \beta_{k+1}$. Now if, in addition, $\{\sigma_l<\infty\}$  finitely often, then necessarily $\limsup_{n\to\infty} \frac{(k+1) U_n^{k+1}}{T_{n}} \le \beta+4\epsilon$, and the conclusion holds. Alternatively, $\{\sigma_l<\infty\}$ i.o., but $\{\tau_l^{u,k+1}<\tau_l^{d,k+1}\}$ finitely often, which clearly gives the desired $\limsup$. The proof is complete. 
\end{proof} 

\section*{Appendix}
\begin{proof}[Proof of Lemma \ref{lem:LDP}]
Let $W = (W_n : n \in \mathds{Z}_+)$ be a simple random walk starting from 0, with increments given by the rule

\begin{equation}
P(W_{n+1} - W_{n} + i) = 
\begin{cases}
p & i = 1 \\
1-p & i = -1 
\end{cases}
\end{equation}
Recall that $pr>1-p$. Therefore $\gamma = E(W_1) = p - (1-p) = \gamma>0$. Let $M_n = \min\{W_{j}: j \leq n\}$. As shown in \cite[Lemma 1]{BA_CLT}, we can construct $T = (T_{n}: n \in \mathds{Z}_+)$ and $M = (W_{n}:n \in \mathds{Z}_+)$ in the same probability space, with $T_n = M_n$ for all $n\in\Z_+$. Let $M_{\infty} = \inf_n M_n$. The transience of $M$ to infinity guarantees that $\zeta_1 \in (0,\infty)$ when $P(\{W_{n} < 0 $ for some $ n> 0\}) = e^{-\zeta_n}$. By the strong Markov property we get $P(M_\infty < -k) = e^{-k\zeta_1}$.

Note that $T_n \leq W_n - M_\infty$. Let $\epsilon' = \epsilon/2$. By the large deviations argument, there exists $c>0$, $\zeta_2 > 0$ such that 

\begin{equation}
P(T_n > n(\gamma+2\epsilon^\prime)) \leq P(W_n-M_\infty > n(\gamma + 2\epsilon^\prime))
\end{equation}

Breaking the above inequality into cases yields

\begin{equation}
\begin{aligned}
P(T_n > n(\gamma+2\epsilon^\prime)) & \leq P(W_n+\epsilon^\prime n > (\gamma + 2\epsilon^\prime), M_\infty > -2\epsilon^\prime) + P(M_\infty < -2\epsilon^\prime) \\
& \leq P(W_n > (\gamma + \epsilon^\prime)) + P(M_\infty < -2\epsilon^\prime) \\
& \leq c(e^{-n\zeta_{n_2}}+e^{-n\zeta_{n_1}})
\end{aligned}
\end{equation}

where, in the last step, large deviations is used to get a bound on $W$ and the bound on $M_\infty$ from earlier is substituted. Since $T_n \geq W_n$, we get

\begin{equation}
\begin{aligned}
P\left(T_n < n(\gamma - 2\epsilon^\prime)\right) & \leq P\left(W_n < n(\gamma - \epsilon^\prime)\right) \\
& \leq ce^{-n\zeta_2}
\end{aligned}
\end{equation}

\end{proof}
\section*{Acknowledgement}
Iddo Ben-Ari would like to thank Rinaldo Schinazi for proposing the model, posing the  problem, fruitful discussions and great ideas, and for Rahul Roy and Hideki Tanemura for bringing their work \cite{roy} to our attention.   
\bibliographystyle{amsalpha}
\bibliography{mybib}
\end{document}